\documentclass[11pt]{article}

\usepackage{amsfonts}
\usepackage{amsmath}

\newcommand{\beq}{\begin{equation}}
\newcommand{\eeq}{\end{equation}}
\newcommand{\bea}{\begin{eqnarray}}
\newcommand{\eea}{\end{eqnarray}}
\newcommand{\beas}{\begin{eqnarray*}}
\newcommand{\eeas}{\end{eqnarray*}}

\newtheorem{theorem}{Theorem}[section]

\newtheorem{proposition}[theorem]{Proposition}
\newtheorem{prop}[theorem]{Proposition}

\newtheorem{lemma}[theorem]{Lemma}
\newtheorem{remark}[theorem]{Remark}
\newtheorem{example}[theorem]{Example}
\newtheorem{examples}[theorem]{Examples}
\newtheorem{foo}[theorem]{Remarks}

\newenvironment{proof}{\addvspace{\medskipamount}\par\noindent{\it
Proof}.}
{\unskip\nobreak\hfill$\Box$\par\addvspace{\medskipamount}}

\newcommand{\bS}{\mathbb S}

\newcommand{\R}{\mathbb R}

\title{The Subelliptic Heat Kernel on the CR sphere}
\author{Fabrice Baudoin\footnote{First author supported in part by
NSF Grant DMS 0907326}, Jing Wang }

\date{Department of Mathematics, Purdue University \\
 West Lafayette, IN, USA}

\begin{document}
\maketitle

\begin{abstract}
We study the heat kernel of the sub-Laplacian $L$  on the CR sphere $\mathbb{S}^{2n+1}$.  An explicit and geometrically meaningful formula for the heat kernel  is obtained. As  a by-product we recover in a simple way the Green function of the conformal sub-Laplacian $-L+n^2$ that was obtained by Geller \cite{G}, and also get an explicit formula for the sub-Riemannian distance. The key point is to work in a set of coordinates that reflects the symmetries coming from the  fibration  $\mathbb{S}^{2n+1} \rightarrow \mathbb{CP}^n$.
\end{abstract}

\tableofcontents
\clearpage

\section{Introduction}
The purpose of this work is to study the heat kernel of the sub-Laplacian of the standard CR structure on $\bS^{2n+1}$.  More precisely, we will be interested in explicit analytic representations and small time asymptotics of the kernel.

 A key point in our study is to take advantage of the radial  symmetries of the fibration $\mathbb{S}^{2n+1} \rightarrow \mathbb{CP}^n$ to introduce coordinates that are adapted to the geometry of the problem. In particular, it is shown that the kernel has a cylindric invariance property and actually only depends on two variables $r,\theta$. The variable $\theta$ is the local fiber coordinate of the fibration and $r$ is a radial coordinate on $\mathbb{CP}^n$ .  In these coordinates the cylindric part of the sub-Laplacian $L$ is the operator
 \[
\tilde{L}=\frac{\partial^2}{\partial r^2}+((2n-1)\cot r-\tan r)\frac{\partial}{\partial r}+\tan^2r\frac{\partial^2}{\partial \theta^2} .
\]
Eigenvalues and eigenvectors for $\tilde{L}$ are computed and as a consequence of the general Minakshisundaram-Pleijel expansion theorem we deduce that the subelliptic heat kernel  (issued from the north pole) can be written
\[
p_t(r, \theta)=\frac{\Gamma(n)}{2\pi^{n+1}}\sum_{k=-\infty}^{+\infty}\sum_{m=0}^{+\infty} (2m+|k|+n){m+|k|+n-1\choose n-1}e^{-\lambda_{m,k}t+ik \theta}(\cos r)^{|k|}P_m^{n-1,|k|}(\cos 2r),
\]
where $P_m^{n-1,|k|}$ is a Jacobi polynomial and $\lambda_{m,k}=4m(m+|k|+n)+2|k|n$. This formula is very useful  to study the long-time behavior of the heat kernel but seems difficult to use in the study of small-time asymptotics or for the purpose of proving upper and lower bounds. In order to derive small-time asymptotics of the kernel, we give another analytic expression for $p_t(r, \theta)$ which is much more geometrically meaningful. This formula is obtained thanks to the observation that the Reeb vector field $T$ of the CR structure of $\bS^{2n+1}$ commutes with the sub-Laplacian. This commutation implies that the subelliptic  heat semigroup $e^{tL}$ can be written $e^{-t T^2} e^{t \Delta}$, where $\Delta$ is the Laplace-Beltrami operator of the standard  Riemannian structure on $\bS^{2n+1}$. The formula we obtain is explicit enough to recover Geller's formula (see  \cite{G}) for the fundamental solution of the conformal sub-Laplacian $-L+n^2$. By using the steepest descent method, it also allows to derive the small-time asymptotics of the heat kernel. A by-product of this small-time asymptotics is a previously unknown explicit formula for the sub-Riemannian distance.

To put things in perspective let us observe that the study of explicit expressions for subelliptic heat kernels has generated a great amount of work (see \cite{A}, \cite{B}  \cite{bauer},  \cite{BB}, \cite{BGG} \cite{Bo}, \cite{Ga}  and the references therein). The motivations for finding explicit formulas are numerous, among them we can cite: sharp constant in functional inequalities  (see \cite{BFM}, \cite{LF}),  computation of the sub-Riemannian metric (see \cite{BB}), sharp upper and lower bounds for the heat kernel (see \cite{BGG},\cite{El}),  and semigroup sub-commutations (see \cite{Li}). 
However, despite such numerous works, very few explicit and tractable formulas are actually known and most of them are restricted to a Lie group framework. The present work gives explicit and tractable expressions that hold  in a natural  sub-Riemannian model.

\section{The sub-Laplacian on $\mathbb{S}^{2n+1}$ }

\subsection{Geometry of the standard CR sphere}

We consider the odd dimensional sphere 
\[
\bS^{2n+1}=\lbrace z=(z_1,\cdots,z_{n+1})\in \mathbb{C}^{n+1}, \| z \| =1\rbrace.
\]
It is a  strictly pseudo convex  CR manifold (see \cite{CR}) whose geometry can be described as follows. There is a natural group action of $\mathbb{S}^1$ on $\bS^{2n+1}$ which is  defined by $$(z_1,\cdots, z_n) \rightarrow (e^{i\theta} z_1,\cdots, e^{i\theta} z_n). $$ The generator of this action shall be denoted by $T$ throughout the paper. We have for all $f \in \mathcal{C}^\infty(\bS^{2n+1})$
\[
Tf(z)=\frac{d}{d\theta}f(e^{i\theta}z)\mid_{\theta=0},
\]
so that
\[
T=i\sum_{j=1}^{n+1}\left(z_j\frac{\partial}{\partial z_j}-\overline{z_j}\frac{\partial}{\partial \overline{z_j}}\right).
\]
This action induces a fibration (circle bundle) from $\bS^{2n+1}$ to the projective complex space $\mathbb{CP}^n$. The vector field $T$ is the Reeb vector field (characteristic direction) of the pseudo-Hermitian contact form
\[
\eta=-\frac{i}{2}\sum_{j=1}^{n+1}(\overline{z_j}dz_j-z_j d\overline{z_j}).
\]

\subsection{The sub-Laplacian}

For $j=1,\cdots,n+1$, let us denote  $$T_j=\frac{\partial}{\partial z_j}-\overline{z_j}\mathcal{S},$$ where $\mathcal{S}=\sum_{k=1}^{n+1}z_k\frac{\partial}{\partial z_k}$, and define the second order differential operator $L$ on $\mathcal{C}^\infty(\bS^{2n+1})$ as follows:
\begin{equation}\label{eq2}
L=2\sum_{j=1}^{n+1}(T_j\overline{T_j}+\overline{T_j}T_j).
\end{equation}

It can be checked that $L$ is the CR sub-Laplacian of the previously described structure (see for instance \cite{BFM}, \cite{CKS}). It is essentially self-adjoint on $\mathcal{C}^\infty(\bS^{2n+1})$ with respect to the uniform measure of $\bS^{2n+1}$ and related to the Laplace-Beltrami operator $\Delta$ of the standard Riemannian structure on $\bS^{2n+1}$ by the formula:
\[
L=\Delta -T^2.
\]
We can observe, a fact which will be important for us, that $L$ and $T$ commute, that is, on smooth functions $TL=LT$.

To study $L$ we now introduce a set of coordinates that takes into account the symmetries of the fibration $\mathbb{S}^{2n+1} \rightarrow \mathbb{CP}^n$. Let $(w_1,\cdots, w_n,\theta)$ be local coordinates for $\bS^{2n+1}$, where $(w_1,\cdots,w_n)$ are the local inhomogeneous coordinates for $\mathbb{CP}^n$ given by $w_j=z_j/z_{n+1}$, and $\theta$ is the local fiber coordinate. i.e., $(w_1, \cdots, w_n)$ parametrizes the complex lines passing through the north pole\footnote{We call north pole the point with complex coordinates $z_1=0,\cdots, z_{n+1}=1$, it is therefore the point with real coordinates $(0,\cdots,0,1,0)$. }, while $\theta$ determines a point on the line that is of unit distance from the north pole. More explicitly, these coordinates are given by
\begin{align}\label{cylinder}
(w_1,\cdots,w_n,\theta)\longrightarrow \left(\frac{w_1e^{i\theta}}{\sqrt{1+\rho^2}},\cdots,\frac{w_ne^{i\theta}}{\sqrt{1+\rho^2}},\frac{e^{i\theta}}{\sqrt{1+\rho^2}} \right),
\end{align}
where $\rho=\sqrt{\sum_{j=1}^{n}|w_j|^2}$, $\theta \in \R/2\pi\mathbb{Z}$, and $w \in \mathbb{CP}^n$. In these coordinates, it is clear that $T=\frac{\partial}{\partial \theta}$. Our goal is now to compute the sub-Laplacian $L$.  In the sequel we denote
\[
\mathcal{R}=\sum_{j=1}^n w_j \frac{\partial}{\partial w_j}.
\] 

\begin{proposition}\label{prop2}
In the coordinates (\ref{cylinder}), we have
\[
L=4(1+\rho^2)\sum_{k=1}^n \frac{\partial^2}{\partial w_k \partial\overline{w_k}}+ 4(1+\rho^2)\mathcal{R} \overline{\mathcal{R}}+\rho^2\ \frac{\partial^2}{\partial \theta^2}-2i(\rho^2+1)(\mathcal{R} -\overline{\mathcal{R}})\frac{\partial}{\partial\theta}.
\]
\end{proposition}

\begin{proof}
From (\ref{eq2}), we know that 
\[
L=2\sum_{k=1}^{n+1}\left(\frac{\partial^2}{\partial z_k\partial\overline{z_k}}+\frac{\partial^2}{\partial \overline{z_k}\partial z_k}\right)-2\sum_{k,j=1}^{n+1}\left(z_j\overline{z_k}\frac{\partial^2}{\partial z_j\partial\overline{z_k}}+\overline{z_j}z_k\frac{\partial^2}{\partial \overline{z_j}\partial z_k}\right)-2n(\mathcal{S}+\overline{\mathcal{S}}).
\]
By using now the diffeomorphism 
\begin{align*}
(w_1,\cdots,w_n,\theta,\kappa)\longrightarrow \left(\frac{\kappa w_1e^{i\theta}}{\sqrt{1+\rho^2}},\cdots,\frac{\kappa w_ne^{i\theta}}{\sqrt{1+\rho^2}},\frac{\kappa e^{i\theta}}{\sqrt{1+\rho^2}} \right),
\end{align*}
and then restrict to the sphere on which we have:
\[
\kappa=1, \frac{\partial }{\partial \kappa}=0,
\]
we compute that on $\bS^{2n+1}$,  for $1\leq k\leq n$
\begin{eqnarray*}
\frac{\partial}{\partial z_k} &=&\sqrt{1+\rho^2}e^{-i\theta}\frac{\partial}{\partial w_k}\\
\frac{\partial}{\partial \overline{z_k}} &=&\sqrt{1+\rho^2}e^{i\theta}\frac{\partial}{\partial \overline{w_k}} \\
\end{eqnarray*}
and
\begin{eqnarray*}
\frac{\partial}{\partial z_{n+1}} &=&-\sqrt{1+\rho^2}e^{-i\theta}\left(\sum_{j=1}^n w_j\frac{\partial}{\partial w_j}-\frac{1}{2i}\frac{\partial}{\partial \theta}\right) \\
\frac{\partial}{\partial \overline{z_{n+1}}}  &=& -\sqrt{1+\rho^2}e^{i\theta}\left(\sum_{j=1}^n\overline{w_j}\frac{\partial}{\partial \overline{w_j}}+\frac{1}{2i}\frac{\partial}{\partial \theta}\right).
\end{eqnarray*}
Tedious but straightforward computations lead then to
\[
\sum_{k,j=1}^{n+1}z_j\overline{z_k}\frac{\partial^2}{\partial z_j\partial\overline{z_k}}=\frac{1}{4}\frac{\partial^2}{\partial \theta^2}-\frac{1}{4i}\frac{\partial}{\partial\theta}
\]
which implies that
\[
\sum_{k,j=1}^{n+1}\left(z_j\overline{z_k}\frac{\partial^2}{\partial z_j\partial\overline{z_k}}+\overline{z_j}z_k\frac{\partial^2}{\partial \overline{z_j}\partial z_k}\right)=\frac{1}{2}\frac{\partial^2}{\partial \theta^2}.
\]
Moreover, it is not hard to compute
\[
\sum_{k=1}^{n+1}\left(\frac{\partial^2}{\partial z_k\partial\overline{z_k}}+\frac{\partial^2}{\partial \overline{z_k}\partial z_k}\right)=2(1+\rho^2)\left(\frac{\partial^2}{\partial w_k\partial\overline{w_k}}+\mathcal{R}\overline{\mathcal{R}}+\frac{1}{4}\frac{\partial^2}{\partial \theta^2}+\frac{1}{2i}(\mathcal{R}-\overline{\mathcal{R}})\frac{\partial}{\partial\theta}\right),
\]
where $\mathcal{R}=\sum_{k=1}^nw_k\frac{\partial}{\partial w_k}$. Finally, it is easy to see that $\mathcal{S}+\overline{\mathcal{S}}=0$. Hence we have the conclusion.
\end{proof}

\begin{remark}\label{eq4}
Notice that $T=\frac{\partial}{\partial\theta}$, and thus the Laplace-Beltrami operator is given by:
\begin{align*}
\Delta&=L+\frac{\partial^2}{\partial\theta^2} \\
 &=4(1+\rho^2)\sum_{k=1}^n \frac{\partial^2}{\partial w_k \partial\overline{w_k}}+ 4(1+\rho^2)\mathcal{R} \overline{\mathcal{R}}+(1+\rho^2)\ \frac{\partial^2}{\partial \theta^2}-2i(\rho^2+1)(\mathcal{R} -\overline{\mathcal{R}})\frac{\partial}{\partial\theta}
\end{align*}
\end{remark}

Due to the symmetries of the fibration $\bS^{2n+1} \to \mathbb{CP}^n$, in the study of the heat kernel, it will be enough  to compute the radial part of $L$ with respect to the cylindrical variables $(\rho,\theta)$.

Let us consider the following second order differential operator
\[
 \tilde{L}=\left(1+\rho^2\right)^2\frac{\partial^2}{\partial \rho^2}+\left(\frac{(2n-1)(1+\rho^2)}{\rho}+(1+\rho^2)\rho\right)\frac{\partial}{\partial \rho}+\rho^2\frac{\partial^2 }{\partial \theta^2},
\]
which is defined on the space $\mathcal{D}$ of smooth functions $ f :  \mathbb{R}_{\ge 0} \times \R/2\pi\mathbb{Z}  \to \mathbb{R}$  that satisfies $\frac{\partial f}{\partial  \rho} =0$  if $\rho =0$. It is  seen that $\tilde{L}$ is essentially self-adjoint on $\mathcal{D}$ with respect to the measure $\frac{\rho^{2n-1}}{(1+\rho^2)^{2n+\frac{1}{2} }} d\rho d\theta$.

\begin{proposition} Let us denote by $\psi$ the map from $\bS^{2n+1} $ to $  \mathbb{R}_{\ge 0} \times \R/2\pi\mathbb{Z} $ such that 
\[
\psi \left(\frac{w_1e^{i\theta}}{\sqrt{1+\rho^2}},\cdots,\frac{w_ne^{i\theta}}{\sqrt{1+\rho^2}},\frac{e^{i\theta}}{\sqrt{1+\rho^2}} \right)=\left(\rho, \theta \right).
\]
 For every  $f \in \mathcal{D}$,  we have
\[
L(f \circ \psi)=(\tilde{L} f) \circ \psi.
\]
\end{proposition}

\begin{proof}
Notice that by symmetries of the fibration, we have
\[
\left(\sum_{k=1}^n\frac{\partial^2}{\partial w_k\partial\overline{w_k}}\right)(f\circ\psi)=\left(\left(\frac{1}{4}\frac{\partial^2}{\partial\rho^2}+\frac{2n-1}{4\rho}\frac{\partial}{\partial \rho}\right)f\right)\circ\psi
\]
and
\[
\mathcal{R}(f\circ\psi)=\overline{\mathcal{R}}(f\circ\psi)=\left(\left(\frac{1}{2}\rho\frac{\partial}{\partial \rho}\right)f\right)\circ\psi.
\]
Together with Proposition \ref{prop2}, we have the conclusion.
\end{proof}

Finally, instead of $\rho$, it will be expedient to introduce the variable $r$ which is defined by $\rho=\tan r$. It is then easy to see that we can write $\tilde{L}$ as
\begin{equation}\label{eq3}
\tilde{L}=\frac{\partial^2}{\partial r^2}+((2n-1)\cot r-\tan r)\frac{\partial}{\partial r}+\tan^2r\frac{\partial^2}{\partial \theta^2}.
\end{equation}
This is the expression of $\tilde{L}$ which shall be the most convenient for us and that is going to be used throughout the paper.

We can observe that  $\tilde{L}$ is symmetric with respect to the measure
\[
d\mu_r=\frac{2\pi^n}{\Gamma(n)}(\sin r)^{2n-1}\cos r drd\theta.
\]
The normalization is chosen in such a way that
\[
\int_{-\pi}^{\pi}\int_0^{\frac{\pi}{2}}d\mu_r=\mu(\bS^{2n+1})=\frac{2\pi^{n+1}}{\Gamma (n+1)}.
\]

\begin{remark}
In the case of $\bS^3$ $(n=1)$, which is isomorphic to the Lie group $SU(2)$, we obtain
\begin{eqnarray*}
\tilde{L}&=&(1+\rho^2)^2\frac{\partial^2}{\partial\rho^2}+\left(\frac{(1+\rho^2)^2}{\rho}\right)\frac{\partial}{\partial \rho}+\rho^2\frac{\partial^2}{\partial\theta^2}\\
              &=&\frac{\partial^2}{\partial r^2}+2\cot 2r\frac{\partial}{\partial r}+\tan^2r\frac{\partial^2}{\partial\theta^2}
\end{eqnarray*}
This coincides with the result in \cite{BB}.
\end{remark}

\begin{remark}
By  Remark  \ref{eq4},  we see that the radial part of the Laplace-Beltrami operator in cylindrical coordinates is
\[
\Delta^r=\frac{\partial^2}{\partial r^2}+((2n-1)\cot r-\tan r)\frac{\partial}{\partial r}+\frac{1}{\cos^2r}\frac{\partial^2}{\partial\theta^2}.
\]
On the other hand, since in the coordinates (\ref{cylinder}), we have $z_{n+1}=\cos r e^{i\theta}$, it is clear that the Riemannian distance form the north pole $\delta$, satisfies
\[
\cos \delta =\cos  r \cos \theta.
\]
An easy calculation shows  that by making the change of variable  $\cos \delta =\cos r \cos \theta$, the operator $\Delta^r$ acts on functions depending only on $\delta$ as
\[
 \frac{\partial^2}{\partial \delta^2}+2n\cot\delta\frac{\partial}{\partial\delta}.
\]
This expression is known to indeed be the expression of the radial part of $\Delta$ in spherical coordinates. 
\end{remark}

\section{The subelliptic heat kernel on $\bS^{2n+1}$}

\subsection{Spectral decomposition of the heat kernel}

From the expression of $L$ above, it is not hard to see that the kernel of $P_t=e^{tL}$ issued from the north pole only comes from the radial part $\tilde{L}$ and depends on $(r, \theta)$. We denote it by $p_t(r, \theta)$.

\begin{prop}
For $t>0$, $r\in[0,\frac{\pi}{2})$, $ \theta\in[-\pi,\pi]$, the subelliptic kernel has the following spectral decomposition:
\[
p_t(r, \theta)=\frac{\Gamma(n)}{2\pi^{n+1}}\sum_{k=-\infty}^{+\infty}\sum_{m=0}^{+\infty} (2m+|k|+n){m+|k|+n-1\choose n-1}e^{-\lambda_{m,k}t+ik \theta}(\cos r)^{|k|}P_m^{n-1,|k|}(\cos 2r),
\]
where $\lambda_{m,k}=4m(m+|k|+n)+2|k|n$ and
\[
P_m^{n-1,|k|}(x)=\frac{(-1)^m}{2^m m!(1-x)^{n-1}(1+x)^{|k|}}\frac{d^m}{dx^m}((1-x)^{n-1+m}(1+x)^{|k|+m})
\]
is a Jacobi polynomial.
\end{prop}

\begin{proof}
The idea is to expand $p_t(r, \theta)$ as a Fourier series in $\theta$. Let
\[
p_t(r, \theta)=\sum_{k=-\infty}^{+\infty} e^{ik\theta}\phi_k(t,r),
\]
be this Fourier expansion. Since $p_t$ satisfies $\frac{\partial p_t}{\partial t}=Lp_t$, we have
\[
\frac{\partial\phi_k}{\partial t}=\frac{\partial^2\phi_k}{\partial r^2}+((2n-1)\cot r-\tan r)\frac{\partial\phi_k}{\partial r}-k^2\tan^2 r\phi_k.
\]
By writing $\phi_k(t,r)$ in the form
\[
\phi_k(t,r)=e^{-2n|k|t}(\cos r)^{|k|}g_k(t, \cos 2r), 
\]
we get 
\begin{equation}\label{eq5}
\frac{\partial g_k}{\partial t}=4\Psi_k(g_k),
\end{equation}
where
\[
\Psi_k=(1-x^2)\frac{\partial^2}{\partial x^2}+[(|k|+1-n)-(|k|+1+n)x]\frac{\partial}{\partial x}.
\]
In fact (\ref{eq5}) is well-known as the Jacobi differential equation, and the eigenvectors are given by
\[
P_m^{n-1,|k|}(x)=\frac{(-1)^m}{2^m m!(1-x)^{n-1}(1+x)^{|k|}}\frac{d^m}{dx^m}((1-x)^{n-1+m}(1+x)^{|k|+m}),
\]
which satisfies that
\[
\Psi_k(P_m^{n-1,|k|})(x)=-m(m+n+|k|)P_m^{n-1,|k|}(x).
\]
Therefore, we obtain the spectral decomposition
\[
p_t(r, \theta)=\sum_{k=-\infty}^{+\infty}\sum_{m=0}^{+\infty} \alpha_{m,k}e^{-[4m(m+|k|+n)+2|k|n]t}e^{ik \theta}(\cos r)^{|k|}P_m^{n-1,|k|}(\cos 2r),
\]
where  the constants $\alpha_{m,k}$'s are to be determined by the initial condition at time $t=0$. 

To compute them, we use the fact that $(P_m^{n-1,|k|}(x)(1+x)^{|k|/2})_{m\geq0}$ is an orthogonal basis of $L^2([-1,1],(1-x)^{n-1}dx)$, i.e.,
\[
\int_{-1}^1 P_m^{n-1,|k|}(x)P_l^{n-1,|k|}(x)(1-x)^{n-1}(1+x)^{|k|}dx=\frac{2^{n+|k|}}{2m+|k|+n}\frac{\Gamma(m+n)\Gamma(m+|k|+1)}{\Gamma(m+1)\Gamma(m+n+|k|)}\delta_{ml}.
\]
For a smooth function $f(r, \theta)$, we can write
\[
f(r, \theta)=\sum_{k=-\infty}^{+\infty}\sum_{m=0}^{+\infty} b_{k,m}e^{ik \theta}P_m^{n-1,|k|}(\cos 2r)\cdot(1+\cos 2r)^{|k|/2}
\]
where the $\lbrace b_{k,m}\rbrace$'s are constants, and thus
\[
f(0,0)=\sum_{k=-\infty}^{+\infty}\sum_{m=0}^{+\infty} b_{k,m}P_m^{n-1,|k|}(1)2^{|k|/2}.
\]
Now, since
\begin{eqnarray*}
& &\int_{-\pi}^\pi\int_0^\frac{\pi}{2} p_t(r, \theta)\overline{f(r, \theta)}d\mu_r \\
&=& \frac{2\pi^n}{\Gamma(n)}\int_{-\pi}^{\pi}\int_0^{\frac{\pi}{2}}p_t(r, \theta)\overline{f(r, \theta)}(\sin r)^{2n-1}\cos rdrd\theta \\
&=& \frac{2\pi^n}{\Gamma(n)} \sum_{k=-\infty}^{+\infty}\sum_{m=0}^{+\infty}\int_{-\pi}^{\pi}\int_0^{\frac{\pi}{2}}\alpha_{m,k}b_{m,k}2^{|k|/2}e^{-\lambda_{m,k}t}|P_m^{n-1,|k|}|^2(\cos r)^{2|k|+1}(\sin r)^{2n-1}drd\theta \\
&=&\frac{2\pi^n}{\Gamma(n)}\sum_{k=-\infty}^{+\infty}\sum_{m=0}^{+\infty}\alpha_{m,k}b_{m,k}e^{-\lambda_{m,k}t}2^{-n-|k|/2-1}\int_{-\pi}^{\pi}\left(\int_0^{\frac{\pi}{2}}2^{n+|k|+1}(\cos r)^{2|k|+1}|P_m^{n-1,|k|}|^2(\sin r)^{2n-1}dr\right)d\theta\\
&=&\frac{2\pi^n}{\Gamma(n)} \sum_{k=-\infty}^{+\infty}\sum_{m=0}^{+\infty}\alpha_{m,k}b_{m,k}e^{-\lambda_{m,k}t}2^{-n-|k|/2-1}(2\pi)||P_m^{n-1,|k|}||^2
\end{eqnarray*}
where $\lambda_{m,k}=4m(m+|k|+n)+2|k|n$, we obtain that
\[
\lim_{t\rightarrow 0}\int_{-\pi}^\pi\int_0^\frac{\pi}{2} p_tfd\mu_r= f(0,0)
\]
as soon as $\alpha_{m,k}=\frac{\Gamma(n)}{2\pi^{n+1}}(2m+|k|+n){m+|k|+n-1\choose n-1}$.
\end{proof}

The spectral decomposition of the heat kernel is explicit and useful but is not really geometrically meaningful. We shall now study another representation of the kernel which is more geometrically meaningful and which will turn out to be much more  convenient  when dealing with the small-time asymptotics problem.
The key idea is to observe that since $\Delta$ and $\frac{\partial}{\partial \theta}$ commutes, by  Remark \ref{eq4} we formally  have
\begin{align}\label{commutation}
e^{tL}=e^{-t\frac{\partial^2}{\partial\theta^2}}e^{t\Delta}.
\end{align}

This gives a way to express the sub-Riemannian heat kernel in terms of the Riemannian one.
Let us recall that the Riemannian heat kernel writes
\begin{equation}\label{eq6}
q_t(\cos\delta)=\frac{\Gamma(n)}{2\pi^{n+1}}\sum_{m=0}^{+\infty}(m+n)e^{-m(m+2n)t}C_m^n(\cos \delta),
\end{equation}
where, as above, $\delta$ is the Riemannian distance from the north pole and
\[
C_m^n(x)=\frac{(-1)^m}{2^m}\frac{\Gamma(m+2n)\Gamma(n+1/2)}{\Gamma(2n)\Gamma(m+1)\Gamma(n+m+1/2)}\frac{1}{(1-x^2)^{n-1/2}}\frac{d^m}{dx^m}(1-x^2)^{n+m-1/2},
\]
is a Gegenbauer polynomial.  Another expression of $q_t (\cos \delta)$ which is useful for the computation of small-time asymptotics is 
\begin{equation}\label{heat_kernel_odd}
q_t (\cos \delta)= e^{n^2t} \left( -\frac{1}{2\pi \sin \delta} \frac{\partial}{\partial \delta} \right)^n V
\end{equation}
where $V(t,\delta)=\frac{1}{\sqrt{4\pi t}} \sum_{k \in \mathbb{Z}} e^{-\frac{(\delta-2k\pi)^2}{4t} }$.

Using the commutation (\ref{commutation}) and the formula $\cos \delta =\cos r \cos \theta$, we then infer the following proposition.
\begin{prop}\label{prop1}
For $t>0$, $r\in[0,\pi/2)$, $ \theta\in[-\pi,\pi]$, 
\begin{equation}\label{eq8}
p_t(r, \theta)=\frac{1}{\sqrt{4\pi t}}\int_{-\infty}^{+\infty}e^{-\frac{(y+i \theta)^2}{4t} }q_t(\cos r\cosh y)dy.
\end{equation}
\end{prop}
\begin{proof}
Let 
\[
h_t(r, \theta)=\frac{1}{\sqrt{4\pi t}}\int_{-\infty}^{+\infty}e^{-\frac{(y+i \theta)^2}{4t} }q_t(\cos r\cosh y)dy,
\]
and $\tilde{L}_0=\frac{\partial^2}{\partial r^2}+((2n-1)\cot r-\tan r)\frac{\partial}{\partial r}$, then we have
\[
\tilde{L}=\tilde{L}_0+\tan^2r\frac{\partial^2}{\partial\theta^2}.
\]
Using the fact that
\[
\frac{\partial}{\partial t}\left(\frac{e^{-\frac{y^2}{4t} }}{\sqrt{4\pi t}}\right)=\frac{\partial^2}{\partial y^2}\left(\frac{e^{-\frac{y^2}{4t}} }{\sqrt{4\pi t}}\right)
\]
and
\[
\frac{\partial}{\partial t}\left(q_t(\cos r\cos \theta) \right)=\Delta(q_t(\cos r\cos\theta)),
\]
we get
\begin{eqnarray*}
& &\tilde{L}h_t(r, \theta) = \left(\tilde{L}_0+\tan^2 r\frac{\partial^2}{\partial \theta^2} \right)h_t(r, \theta)\\
                 &=& \int_{-\infty}^{+\infty}\left[ \tilde{L}_0\left(\left(\frac{e^{-\frac{(y+i \theta)^2}{4t}}}{\sqrt{4\pi t}}\right)q_t(\cos r\cosh y)\right)+\tan^2r\frac{\partial^2}{\partial \theta^2}\left(\frac{e^{-\frac{(y+i \theta)^2}{4t}}}{\sqrt{4\pi t}}\right)q_t(\cos r\cosh y)  \right]dy \\
                 &=& \int_{-\infty}^{+\infty}\left[\left(\Delta-\frac{1}{\cos^2r}\frac{\partial^2}{\partial \theta^2}\right)\left(\left(\frac{e^{-\frac{(y+i \theta)^2}{4t}}}{\sqrt{4\pi t}}\right)q_t(\cos r\cosh y)\right) +\tan^2r\frac{\partial^2}{\partial \theta^2}\left(\frac{e^{-\frac{(y+i \theta)^2}{4t}}}{\sqrt{4\pi t}}\right)q_t(\cos r\cosh y)\right]dy\\
                 &=& \int_{-\infty}^{+\infty} \left[\left(\frac{e^{-\frac{(y+i \theta)^2}{4t}}}{\sqrt{4\pi t}}\right)\frac{\partial q_t}{\partial t}+\frac{1}{\cos^2 r}\frac{\partial}{\partial t}\left(\frac{e^{-\frac{(y+i \theta)^2}{4t}}}{\sqrt{4\pi t}}\right)q_t-\tan^2r\frac{\partial}{\partial t}\left(\frac{e^{-\frac{(y+i \theta)^2}{4t}}}{\sqrt{4\pi t}}\right)q_t\right]dy \\
                 &=& \frac{\partial }{\partial t}h_t(r, \theta)
\end{eqnarray*}
On the other hand, it suffices to check the initial condition for functions of the form $f(r, \theta)=e^{i\lambda \theta}g(r)$ where $\lambda\in\mathbb{R}$ and $g$ is smooth. We observe that
\[
\int_0^{\frac{\pi}{2}}\int_0^{2\pi}h_t(r, \theta)f(r, \theta)d\mu_r=e^{t\lambda^2}(e^{t\Delta_r} g)(0).
\]
Thus $h_t(r, \theta)$ is the desired subelliptic heat kernel.
\end{proof}

\begin{prop}
For $\lambda\in\mathbb{C}$, $\mathbf{Re}\lambda>0$, $r\in[0,\pi/2)$, $ \theta\in[-\pi,\pi]$,
\begin{equation}\label{eq7}
\int_0^{+\infty}p_t(r, \theta)e^{-n^2t-\frac{\lambda}{t}}dt=\int_{-\infty}^{\infty}\frac{\Gamma(n)dy}{2^{n+2}\pi^{n+1}\left(\cosh \sqrt{y^2+4\lambda}-\cos r\cos( \theta+iy)\right)^{n}}
\end{equation}
\end{prop}
\begin{proof}
Since
\begin{eqnarray*}
& & \int_0^{+\infty}p_t(r, \theta)e^{-n^2t-\frac{\lambda}{t}}dt=\int_0^{+\infty}e^{-\frac{\lambda}{t}}p_t(r, \theta)e^{-n^2t}dt\\
&=& \frac{1}{\sqrt{4\pi}}\int_{-\infty}^{+\infty}\int_0^{+\infty}e^{-n^2t-\frac{y^2+4\lambda}{4t}}q_t(\cos r\cos( \theta+iy))\frac{dt}{\sqrt{t}}dy
\end{eqnarray*}
We want to compute 
\[
\int_0^{+\infty}e^{-n^2t-\frac{y^2+4\lambda}{4t}}q_t(\cos r\cos( \theta+iy))\frac{dt}{\sqrt{t}}.
\]
Notice that 
\begin{eqnarray*}
\int_0^{+\infty}e^{-n^2t-\frac{y^2+4\lambda}{4t}}e^{t\Delta}\frac{dt}{\sqrt{t}}
&=& \int_0^{+\infty}e^{-\frac{y^2+4\lambda}{4t}}e^{-t(n^2-\Delta)}\frac{dt}{\sqrt{t}} \\
&=& \frac{\sqrt{\pi}}{\sqrt{n^2-\Delta}}e^{-\sqrt{y^2+4\lambda}\sqrt{n^2-\Delta}}
\end{eqnarray*}
However, by the result in Taylor [\cite{T2}, pp. 95], we have that
\[
A^{-1}e^{-tA}=\frac{1}{2}\pi^{-(n+1)}\Gamma(n)(2\cosh t-2\cos \delta)^{-n},
\]
where $A=\sqrt{n^2-\Delta}$ and $\cos\delta=\cos r\cos( \theta+iy)$. Plug in $t=\sqrt{y^2+4\lambda}$, we obtain
\[
\frac{1}{\sqrt{n^2-\Delta}}e^{-\sqrt{y^2+4\lambda}\sqrt{n^2-\Delta}}=\frac{\Gamma(n)}{2^{n+1}\pi^{n+1}\left(\cosh \sqrt{y^2+4\lambda}-\cos r\cos( \theta+iy)\right)^{n}}
\]
hence complete the proof.
\end{proof}

We can deduce the Green function of $-L+n^2$ immediately from the above proposition.

\begin{prop}
The Green function of the conformal sub-Laplacian $-L+n^2$ is given by 
\begin{equation*}
G(r, \theta)=\frac{\Gamma\left(\frac{n}{2}\right)^2}{8\pi^{n+1}(1-2\cos r\cos \theta+\cos^2r)^{n/2}} 
\end{equation*}
\end{prop}
\begin{proof}
Let us assume $r\not =0$, $ \theta\not=0$, and let $\lambda\rightarrow 0$ in (\ref{eq7}), we have
\begin{eqnarray*}
G(r, \theta) &=& \frac{\Gamma(n)}{2^{n+2}\pi^{n+1}}\int_{-\infty}^{+\infty}\frac{dy}{(\cosh y-\cos r\cos( \theta+iy))^n} \\
            &=& \frac{\Gamma(n)}{2^{n+2}\pi^{n+1}}\int_{-\infty}^{+\infty}\frac{dy}{(\cosh y(1-\cos r\cos \theta)-i\cos r\sin \theta\sinh y)^n} \\
            &=& \frac{\Gamma(n)}{2^{n+2}\pi^{n+1}}\frac{1}{(1-2\cos r\cos \theta+\cos^2r)^{n/2}}\int_{-\infty}^{+\infty}\frac{dy}{(\cosh y)^n}        
\end{eqnarray*} 
Notice that 
\begin{eqnarray*}
\int_{-\infty}^{+\infty}\frac{dy}{(\cosh y)^n}&=&\frac{\pi(n-2)!!}{(n-1)!!} \qquad \mbox{when $n>0$ is odd} \\
                                              &=&\frac{2(n-2)!!}{(n-1)!!}\qquad \mbox{when $n>0$ is even}
\end{eqnarray*}
where $n!!$ denotes the double factorial such that
\begin{eqnarray*}
n!!&=&n\cdot(n-2)\cdots5\cdot 3\cdot 1 \qquad\mbox{when $n>0$ is odd} \\
    &=& n\cdot(n-2)\cdots6\cdot 4\cdot 2 \qquad\mbox{when $n>0$ is even} 
\end{eqnarray*}
Moreover, since
\begin{eqnarray*}
\Gamma(n)\frac{(n-2)!!}{(n-1)!!}&=&\frac{2^{n-1}}{\pi}\Gamma\left(\frac{n}{2}\right)^2 \qquad \mbox{when $n>0$ is odd} \\
                                              &=&2^{n-2}\Gamma\left(\frac{n}{2}\right)^2\qquad \mbox{when $n>0$ is even}
\end{eqnarray*}
we obtain
\[
\Gamma(n)\int_{-\infty}^{+\infty}\frac{dy}{(\cosh y)^n}=2^{n-1}\Gamma\left(\frac{n}{2}\right)^2 \qquad\mbox{for all $n\in\mathbb{Z}_{> 0}$}
\]
This implies our conclusion.
\end{proof}

\begin{remark}
This agrees with the result by Geller in \cite{G}.
\end{remark}

\subsection{Asymptotics of the subelliptic heat kernel in small times}

First, we study the asymptotics of the subelliptic heat kernel when $t\rightarrow 0$ on the cut-locus of $0$.  From (\ref{heat_kernel_odd}), we already know the following asymptotics of the heat kernel
\begin{align}\label{eq9}
q_t(\cos\delta)=\frac{1}{(4\pi t)^{n+\frac{1}{2}}}\left(\frac{\delta}{\sin\delta}\right)^ne^{-\frac{\delta^2}{4t}}\left(1+\left(n^2-\frac{n(n-1)(\sin\delta-\delta\cos\delta)}{\delta^2\sin\delta}\right)t+O(t^2)\right),
\end{align}
where $\delta \in [0,\pi)$. Here $\delta$ is the Riemannian distance. Together with (\ref{eq8}), we first deduce the following small-time-asymptotics of the subelliptic heat kernel on the diagonal.
\begin{prop}
When $t\to 0$, 
\[
p_t(0,0)=\frac{1}{(4\pi t)^{n+1}}(A_n+B_nt+O(t^2)),
\]
where $A_n=\int_{-\infty}^{\infty}\frac{y^n}{(\sinh y)^n}dy$ and $B_n=\int_{-\infty}^{\infty}\frac{y^n}{(\sinh y)^n}\left(1+\left(n^2+\frac{n(n-1)(\sinh y-y\cosh y)}{y^2\sinh y}\right)\right)dy$.
\end{prop}
\begin{proof}
We know  that
\begin{eqnarray*}
p_t(0,0)= \frac{1}{\sqrt{4\pi t}}\int_{-\infty}^{\infty}e^{-\frac{y^2}{4t}}q_t(\cosh y)dy.
\end{eqnarray*}
Plug in (\ref{eq9}), we have the desired small time asymptotics.
\end{proof}
\begin{prop}
For $ \theta\in(0,\pi)$, $t\rightarrow 0$,
\[
p_t(0, \theta)=\frac{ \theta^{n-1}}{2^{3n}t^{2n}(n-1)!}e^{-\frac{2\pi \theta- \theta^2}{4t}}(1+O(t))
\]
\end{prop}
\begin{proof}
Let $ \theta\in(0,\pi)$, we have
\[
p_t(0, \theta)=\frac{1}{\sqrt{4\pi t}}\int_{-\infty}^{\infty}e^{-\frac{y^2}{4t}}q_t(\cosh(y-i \theta))dy
\] 
By Cauchy integral theorem, this is the same as integrating along the horizontal line in the complex plane by shifting up $i\theta$ from the real axis. i.e.,
\[
p_t(0, \theta)=\frac{1}{\sqrt{4\pi t}}\int_{-\infty}^{\infty}e^{-\frac{(y+i\theta)^2}{4t}}q_t(\cosh y)dy
\]
Moreover, by (\ref{eq9}), we know that
\[
q_t(\cosh y)\sim _{t\rightarrow 0}\frac{1}{(4\pi t)^{n+\frac{1}{2}}}\left(\frac{y}{\sinh y}\right)^ne^{\frac{y^2}{4t}}.
\]
This gives 
\[
p_t(0, \theta) \sim_{t\rightarrow 0}  \frac{e^{\frac{ \theta^2}{4t} }}{(4\pi t)^{n+1}}\int_{-\infty}^{\infty}\frac{y^n}{(\sinh y)^n}e^{-\frac{iy \theta}{2t}}dy  
\]
By the residue theorem, we get
\begin{eqnarray*}
\int_{-\infty}^{\infty}\frac{y^n}{(\sinh y)^n}e^{-\frac{iy \theta}{2t}}dy 
&=& -2\pi i\sum_{k\in\mathbb{Z^+}} \mathbf{Res}\left( \frac{e^{-\frac{iy \theta}{2t}}y^n}{(\sinh y)^n},-k\pi i\right)\\
&=& -2\pi i\sum_{k\in\mathbb{Z^+}} \frac{1}{(n-1)!}\frac{\partial^{n-1}}{\partial y^{n-1}}\left[\frac{e^{-\frac{iy\theta}{2t}}2^ny^n(y+k\pi i)^n}{(e^y-e^{-y})^n} \right]_{y=-k\pi i}
\end{eqnarray*}
Write $W(y)=\frac{(y+k\pi i)^n}{(e^y-e^{-y})^n}$, $W(y)$ is analytic around $-k\pi i$, and satisfies
\[
W(-k\pi i)=\frac{1}{(-1)^{kn}2^n}.
\]
Hence the residue is 
\[
\frac{1}{(n-1)!}\frac{\partial^{n-1}}{\partial y^{n-1}}\left[e^{-\frac{iy\theta}{2t}}2^ny^nW(y) \right]_{y=-k\pi i}.
\]
This is a product of $e^{-\frac{iy\theta}{2t}}$ and a polynomial of degree $n-1$ in $1/t$. We are only interested in the leading term which plays the dominant role when $t\rightarrow 0$. Thus we have the equivalence 
\[
\frac{-2\pi i}{(n-1)!}\frac{\partial^{n-1}}{\partial y^{n-1}}\left[e^{-\frac{iy\theta}{2t}}2^ny^nW(y) \right]_{y=-k\pi i}
\sim_{t\rightarrow 0}
\frac{(-1)^{kn+n}\pi^{n+1}\theta^{n-1}}{(n-1)!2^{n-2}t^{n-1}}e^{-\frac{k\pi \theta}{2t}}
\]
At the end, we conclude
\[ 
p_t(0, \theta)\sim_{t\rightarrow 0} \frac{e^{\frac{ \theta^2}{4t} }}{(4\pi t)^{n+1}}\sum_{k\in\mathbb{Z}^+}\frac{(-1)^{kn+n}\pi^{n+1} \theta^{n-1}}{(n-1)!2^{n-2}t^{n-1}}e^{-\frac{k\pi \theta}{2t}} ,
\]
that is
\[ 
p_t(0, \theta)=\frac{ \theta^{n-1}}{2^{3n}t^{2n}(n-1)!}e^{-\frac{2\pi \theta- \theta^2}{4t}}(1+O(t))
\]
\end{proof}

Come to the points that do not lie on the cut-locus, i.e., $r\not=0$. First we deduce the case for $(r,0)$.
\begin{prop}
For $r\in(0,\frac{\pi}{2})$,  we have
\[
p_t(r,0)\sim \frac{e^{-\frac{r^2}{4t}}}{(4\pi t)^{n+\frac{1}{2}}}\left(\frac{r}{\sin r}\right)^n\sqrt{\frac{1}{1-r\cot r}}
\]
as $t\rightarrow 0$.
\end{prop}
\begin{proof}
By proposition \ref{prop1},
\[
p_t(r,0)=\frac{1}{\sqrt{4\pi t}}\int_{-\infty}^{\infty}e^{-\frac{y^2}{4t}}q_t(\cos r\cosh y)dy,
\]
together with (\ref{eq9}), it gives that 
\[
p_t(r,0)\sim_{t\rightarrow 0}\frac{1}{(4\pi t)^{n+1}}(J_1(t)+J_2(t))
\]
where 
\[
J_1(t)=\int_{\cosh y\leq\frac{1}{\cos r}}e^{-\frac{y^2+(\arccos (\cos r\cosh y))^2}{4t}}
\left(\frac{\arccos (\cos r\cosh y)}{\sqrt{1-\cos^2r\cosh^2y}} \right)^ndy
\]
and
\[
J_2(t)=\int_{\cosh y\geq\frac{1}{\cos r}}e^{-\frac{y^2-(\cosh^{-1} (\cos r\cosh y))^2}{4t}}
\left(\frac{\cosh^{-1} (\cos r\cosh y)}{\sqrt{\cos^2r\cosh^2y-1}} \right)^ndy
\]
We can analyze $J_1(t)$ and $J_2(t)$ by Laplace method. 
First, notice that in $\left[-\cosh^{-1}(\frac{1}{\cos r}),\cosh^{-1}(\frac{1}{\cos r})\right]$, 
\[
f(y)=y^2+(\arccos (\cos r\cosh y))^2
\]
has a unique minimum at $y=0$, where
\[
f^{''}(0)=2(1-r\cot r).
\] 
Hence by Laplace method, we can easily obtain that
\[
J_1(t)\sim_{t\rightarrow 0}e^{-\frac{r^2}{4t}}\left(\frac{r}{\sin r}\right)^n\sqrt{\frac{4\pi t}{1-r\cot r}}
\]
On the other hand, on $\left(-\infty,-\cosh^{-1}(\frac{1}{\cos r})\right)\cup\left(\cosh^{-1}(\frac{1}{\cos r}),+\infty\right)$, the function 
\[
g(y)=y^2-(\cosh^{-1} (\cos r\cosh y))^2
\]
has no minimum, which implies that $J_2(t)$ is negligible with respect to $J_1(t)$ in small $t$. Hence the conclusion.
\end{proof}

We can now extend the result to the $\theta\not=0$ case by applying the steepest descent method. 
\begin{lemma}
For $r\in(0,\frac{\pi}{2})$, $\theta\in[-\pi,\pi]$, 
\[
f(y)=(y+i\theta)^2+(\arccos (\cos r \cosh y))^2
\]
defined on the strip $|\mathbf{Re}(y)|<\cosh^{-1}\left( \frac{1}{\cos r}\right)$ has a critical point at $i\varphi(r,\theta)$, where $\varphi(r,\theta)$ is the unique solution in $[-\pi,\pi]$ to the equation
\[
\varphi(r,\theta)+\theta=\cos r\sin\varphi(r,\theta)\frac{\arccos(\cos \varphi(r,\theta)\cos r)}{\sqrt{1-\cos^2r\cos^2\varphi(r,\theta)}}.
\]
\end{lemma}
\begin{proof}
Let $u=\cos r\cos\varphi$,
\[
\frac{\partial}{\partial\varphi}\left(\varphi-\cos r\sin\varphi\frac{\arccos(\cos\varphi\cos r)}{\sqrt{1-\cos^2 r\cos^2\varphi}}\right)=\frac{\sin^2 r}{1-u(r,\theta)^2}\left(1-\frac{u(r,\theta)\arccos u(r,\theta)}{\sqrt{1-u^2(r,\theta)}}
\right)\]
is positive, thus $\theta=\varphi-\cos r\sin\varphi\frac{\arccos(\cos\varphi\cos r)}{\sqrt{1-\cos^2 r\cos^2\varphi}}$ is a bijection from $[-\pi, \pi]$ onto itself, hence the uniqueness.
\end{proof}
Moreover, observe that at $\varphi(r,\theta)$, 
\[
f^{''}(i\varphi(r,\theta))=\frac{2\sin^2 r}{1-u(r,\theta)^2}\left(1-\frac{u(r,\theta)\arccos u(r,\theta)}{\sqrt{1-u^2(r,\theta)}} \right),
\]
is positive, where $u(r,\theta)=\cos r\cos\varphi(r,\theta)$. 

By using the steepest descent method we can deduce 
\begin{prop}
Let  $r\in(0,\frac{\pi}{2})$, $\theta\in[-\pi,\pi]$. Then when $t\to 0$,
\[
p_t(r,\theta)\sim_{t\rightarrow 0}\frac{1}{(4\pi t)^{n+\frac{1}{2}}\sin r}\frac{(\arccos u(r,\theta))^n}{\sqrt{1-\frac{u(r,\theta)\arccos u(r,\theta)}{\sqrt{1-u^2(r,\theta)}}}}\frac{e^{-\frac{(\varphi(r,\theta)+\theta)^2\tan^2 r}{4t\sin^2(\varphi(r,\theta))}}}{(1-u(r,\theta)^2)^{\frac{n-1}{2}}},
\]
where $u(r,\theta)=\cos r\cos\varphi(r,\theta)$.
\end{prop}

\begin{remark}
By symmetry, the sub-Riemannian distance from the north pole to any point on $\bS^{2n+1}$ only depends on $r$ and $\theta$. If we denote it by $d(r,\theta)$, then from the previous propositions,
\item[(1)]
For $\theta\in [-\pi,\pi]$,
\[
d^2(0,\theta)=2\pi |\theta|-\theta^2
\]
\item[(2)]
For $\theta\in [-\pi,\pi]$, $r\in\left(0,\frac{\pi}{2}\right)$,
\[
d^2(r,\theta)=\frac{(\varphi(r,\theta)+\theta)^2\tan^2 r}{\sin^2(\varphi(r,\theta))}
\]
In particular,  the sub-Riemannian diameter of $\bS^{2n+1}$ is $\pi$. For a study of the sub-Riemannian geodesics on $\bS^{2n+1}$, we refer to \cite{MM}.
\end{remark}


\begin{thebibliography}{10}
\bibitem{A} Agrachev A., Boscain U., Gauthier J.P., Rossi F. \textit{The intrinsic hypoelliptic Laplacian and its heat kernel on unimodular Lie groups}, Journal of Functional Analysis, Vol. 256, 8, (2009), 2621-2655.
\bibitem{B} Barilari D., \textit{Trace heat kernel asymptotics in 3D contact sub-Riemannian geometry}, arXiv:1105.1285
\bibitem{BB} Baudoin, F., Bonnefont, M. \textit{The subelliptic heat kernel on $SU(2)$: representations, asymptotics and gradient bounds,} Math. Z. \textbf{263} (2009) 647-672
\bibitem{bauer} Bauer R.O. , \textit{Analysis of the horizontal Laplacian for the Hopf fibration}, Forum Mathematicum, (2005), Vol. 17, 6,  903--920
\bibitem{BGG} Beals, R., Gaveau, B., Greiner, P. C. \textit{Hamilton-Jacobi theory and the heat kernel on Heisenberg groups,} J. Math. Pures Appl. 79, 7 (2000) 633-689
\bibitem{BFM} Branson, T.P., Fontana, L., Morpurgo, C. \textit{Moser-Trudinger and Beckner-Onofri's inequalities on the CR sphere,} arXiv:0712.3905v3
\bibitem{Bo} Bonnefont M., \textit{The subelliptic heat kernel on SL(2,R) and on its universal covering:  integral representations and some functional inequalities}. To appear in Potential analysis.
\bibitem{CKS} Cowling, M. G., Klima, O., Sikora, A. \textit{Spectral Multipliers for the Kohn sublaplacian on the sphere in $\mathbb{C}^n$,} Trans. Amer. Math. Soc. \textbf{363} No. 2 (2011) 611-631
\bibitem{CR} Dragomir S., Tomassini G., \textit{Differential geometry and analysis on CR manifolds}, Birkh\"auser, Vol. 246, 2006.
\bibitem{El} Eldredge, N. \textit{Gradient estimates for the subelliptic heat kernel on H-type groups}. J. Funct. Anal. 258 (2010), pp. 504-533. 
\bibitem{Ga} Gaveau B., \textit{Principe de moindre action, propagation de la chaleur et estim\'eees sous elliptiques sur certains groupes nilpotents}, Acta Math.
Volume 139, Number 1, 95-153, (1977).
\bibitem{G} Geller, D., \textit{The Laplacian and the Kohn Laplacian for the sphere}, J. Differential Geometry. \textbf{15} (1980) 417-435
\bibitem{Li} Li H.Q., \textit{Estimation optimale du gradient du semi-groupe de la chaleur sur le groupe de Heisenberg.}, Jour. Func. Anal. , 236, pp 369-394, (2006).
\bibitem{LF} Lieb, E., Frank R., \textit{ Sharp constants in several inequalities on the Heisenberg group.}, arXiv:1009.1410, To appear in Ann. Math.
\bibitem{MM} Molina, M, Markina I. : \textit{Sub-Riemannian geodesics and heat operator on odd dimensional spheres},  arXiv:1008.5265
\bibitem{T2}Taylor, M. E. \textit{Partial differential equations. II,} Applied Mathematical Sciences \textbf{116}, Springer-Verlag, New York (1996) \end{thebibliography}
\end{document}